\documentclass{amsart}

\usepackage{amssymb} 
\usepackage{amsmath} 
\usepackage{latexsym} 
\usepackage{amsfonts}
\usepackage{amsthm}  
\usepackage{mathrsfs}
\usepackage{verbatim} 
\usepackage{stmaryrd} 
\usepackage{color}  
\usepackage[pdftex,colorlinks=true, pdfstartview=FitV, linkcolor=blue, citecolor=blue, urlcolor=blue]{hyperref} 
\usepackage[all,2cell,ps]{xy}
\usepackage{tikz}
\usepackage{lscape} 

\numberwithin{equation}{section}

\theoremstyle{plain} 
\newtheorem{thm}{Theorem}[section]

\newtheorem{prop}[thm]{Proposition}
\newtheorem{cor}[thm]{Corollary}
\newtheorem{conj}[thm]{Conjecture}
\newtheorem*{conj*}{Conjecture}

\newtheorem{quest}[thm]{Question}

\theoremstyle{definition}
\newtheorem{defn}[thm]{Definition}
\newtheorem{example}[thm]{Example}  

\theoremstyle{remark}  
\newtheorem{rem}[thm]{Remark}

\newtheorem*{observ*}{Observation}

\newtheorem*{claim*}{Claim}

\newtheoremstyle{editorialNotes}
	{10mm} 
	{10mm}
	{\slshape}
	{-80pt}
	{\bfseries}
	{}
	{10mm}
	{}

\theoremstyle{editorialNotes}

\newcommand{\ds}{\displaystyle}
\newcommand{\ol}{\overline}

\newcommand{\vl}{\;\vert\;}

\newcommand{\C}{\mathbb{C}}

\newcommand{\kk}{\Bbbk}
\newcommand{\m}{\mathfrak{m}}

\newcommand{\bb}{\mathbf{b}}

\newcommand{\CM}{Cohen-Macaulay\ }

\newcommand{\ls}{\leqslant}
\newcommand{\gs}{\geqslant}

\newcommand{\SuSp}{super-stretched}

\newcommand{\lst}[2]{#1_1,\dots,#1_{#2}}

\newcommand{\dfn}[1]{\textsf{#1}\index{#1}}

\author[B. Stone]{Branden Stone}
\thanks{The author was partially funded by the NSF grant, Kansas Partnership for Graduate Fellows in K-12 Education (DGE-0742523).}

\title[Non-Gorenstein, graded countable Cohen-Macaulay type]{Non-Gorenstein isolated singularities of graded countable Cohen-Macaulay type}
\address{Branden Stone, Mathematics Program, Bard College/Bard Prison Initiative, P.O. Box 5000, Annandale-on-Hudson, NY 12504}
\email{bstone@bard.edu}

\begin{document}

\maketitle 
	
\begin{abstract}
	In this paper we show a partial answer the a question of C. Huneke and G. Leuschke (2003): Let $R$ be a standard graded \CM ring of graded countable \CM representation type, and assume that $R$ has an isolated singularity. Is $R$ then necessarily of graded finite \CM representation type? In particular, this question has an affirmative answer for standard graded non-Gorenstein rings as well as for standard graded Gorenstein rings of minimal multiplicity. Along the way, we obtain a partial classification of graded \CM rings of graded countable \CM type.
\end{abstract}


\section{Motivation and Introduction}\label{ch1}

	For motivational purposes, we first consider local \CM rings.  Throughout, a local \CM ring is said to have \dfn{finite \CM type} (respectively, \dfn{countable \CM type}) if it has only finitely (respectively, countably) many isomorphism classes of indecomposable maximal \CM modules.  

In \cite{auslander86}, M. Auslander showed that  local rings of finite \CM type must have an isolated singularity.  It was later shown in \cite{auslander89} that this was true for standard graded rings as well.  Concerning countable \CM type, it was shown in the excellent case by C. Huneke and G. Leuschke that a local ring of countable \CM type does not necessarily have an isolated singularity, but the singular locus has dimension at most one \cite{huneke03}(the graded version of this is found in \cite{mythesis}).  Given this result, it is natural to try and find a local \CM ring of countable (and infinite) \CM type with an isolated singularity.  Examining the known examples countable type rings with an isolated singularity, C. Huneke and G. Leuschke found they were actually finite type!  This led to the following question:
	\begin{quest}[\cite{huneke03}]\label{q1}
		Let $R$ be a complete local Cohen-Macaulay ring of countable \CM representation type, and assume that $R$ has an isolated singularity. Is $R$ then necessarily of finite \CM representation type?  
	\end{quest}
In this paper, we are mainly concerned with the graded version of this question (see Question \ref{q2}).  In Corollary \ref{main:cor}, we are able to show a positive answer for standard graded non-Gorenstein rings as well as for graded Gorenstein rings of minimal multiplicity.  The main strategy is to classify the rings of graded countable \CM type and compare this list with a the classification of graded finite \CM type in \cite{eisenbud88}. 

The sections of this paper break the problem up into cases based on the dimension of the ring, with Section \ref{sec:gor} discussing the issues surrounding Gorenstein rings.  The results are summarized in Section \ref{sec:sumup}.

\subsection{Preliminaries and known results}

We say that a ring $R$ is \dfn{standard graded}\label{def:stgrade} if, as an abelian group, it has a decomposition $R = \bigoplus_{i \gs 0} R_i$ such that $R_iR_j \subseteq R_{i+j}$ for all $i,j\gs 0$, $R = R_0[R_1]$, and $R_0$ is a field.  Further, we will always assume that a standard graded ring is Noetherian. Unless otherwise stated, we will denote $(R,\m,\kk)$ by the standard graded ring with $\m$ being the irrelevant maximal ideal, that is $\m = \sum_{i =1}^\infty R_i$, and $\kk := R_0 = R/\m$ being an uncountable algebraically closed field of characteristic zero.  We let $e(R)$ be the multiplicity of the maximal ideal $\m$.  When $R$ is a $d$-dimensional \CM ring, if $e(R)= \dim_\kk(\m/\m^2) - \dim R + 1 $ then $R$ is said to have \dfn{minimal multiplicity}.  In the case when $\kk$ is infinite, the following are equivalent: 
\begin{itemize}\label{min-mult-equiv} 
	\item[(i)] $R$ has minimal multiplicity;
	\item[(ii)] there exists a regular sequence $x_1,\ldots,x_d$ such that $\m^2 = (x_1,\ldots,x_d)\m$;
	\item[(iii)] the $h$-vector of $R$ is of the form $(1,n)$.
\end{itemize}
Here we define the \dfn{$h$-vector} of $R$ to be the sequence $\ds \left(\dim_\kk(R/(\lst x d))_i \right)_{i \gs 0}$ where $\lst x d$ is a homogeneous regular sequence of degree one.

A standard graded \CM ring $(R,\m,\kk)$ has \dfn{graded finite \CM type} (respectively, \dfn{graded countable \CM type}) if it has only finitely (respectively, countably) many indecomposable, maximal \CM modules up to a shift in degree.

The graded version of Question \ref{q1} is as follows.
\begin{quest}\label{q2}
	Let $R$ be a standard graded \CM ring of graded countable \CM representation type, and assume that $R$ has an isolated singularity. Is $R$ then necessarily of graded finite \CM representation type?  
\end{quest}

\begin{rem}\label{rem:local-to-graded}
	It is known that if the completion of a standard graded ring $R$ with respect to the maximal ideal has finite (respectively countable) type, then $R$ must have graded finite (respectively countable) type (see \cite[Prop 8 and 9]{auslander89} or \cite[Corollary 2.5]{stoneSS}).  Hence any affirmative answer to Question \ref{q1} can be passed to a positive answer of Question \ref{q2}.
\end{rem}

With this in mind, Question \ref{q2} has a positive answer if the ring is a hypersurface.  In particular, it was shown  in \cite{knorrer87,buchweitz87} that if $R$ is a complete hypersurface containing an algebraically closed field $\kk$ (of characteristic different from 2), then $R$ is of finite \CM type if and only if $R$ is the local ring of a simple hypersurface singularity in the sense of \cite{arnold74}.  For example, if we let $\kk = \C$, then $R$ is one of the complete ADE singularities over $\C$.  That is, $R$ is isomorphic to $\kk\llbracket x, y, z_2,\ldots, z_r\rrbracket/( f )$, where $f$ is one of the following polynomials:
\begin{align*}
(A_n): &\  x^{n+1}+y^2+z_2^2+\cdots+z_r^2 ,\  n \gs 1; \\
(D_n): &\  x^{n-1}+xy^2+z_2^2+\cdots+z_r^2 , \ n \gs 4; \\
(E_6): &\  x^4+y^3+z_2^2+\cdots+z_r^2; \\
(E_7): &\  x^3y+y^3+z_2^2+\cdots+z_r^2; \\
(E_8): &\  x^5+y^3+z_2^2+\cdots+z_r^2.
\end{align*}
It was further shown in \cite{buchweitz87} that a complete hypersurface singularity over an algebraically closed uncountable field $\kk$ has (infinite) countable \CM type if and only if it is isomorphic to one of the following:
\begin{align}
	(A_\infty): &\ \kk\llbracket x, y, z_2,\ldots, z_r\rrbracket /( y^2+z_2^2+\cdots+z_r^2 ); \label{eq1}\\ 
	(D_\infty): &\ \kk\llbracket x, y, z_2,\ldots, z_r\rrbracket /( xy^2+z_2^2+\cdots+z_r^2 ). \label{eq2}
\end{align}	
As both $(A_\infty)$ and $(D_\infty)$ are non-isolated singularities, if a hypersurface has an isolated singularity and is countable type then it must have finite type.  

In 2011, R. Karr and R. Wiegand \cite[Theorem 1.4]{karr11} showed the one dimensional case as well. This was under the assumption that the integral closure of the ring $R$ is finitely generated as an $R$-module.  We recover these results in Section \ref{dim1}.

\section{Zero Dimensional Rings}

	It is well known that a zero dimensional equicharacteristic local ring $R$ is a hypersurface if and only if $R$ is of finite \CM type \cite[Satz 1.5]{herzog78}.  We show the graded countable analog to this statement in Proposition \ref{prop:gct-0-dim}.

\begin{prop}\label{prop:gct-0-dim}
	Let $(R,\m,\kk)$ be a 0-dimensional standard graded \CM ring. Further assume that $\kk$ is an uncountable field. Then the following are equivalent:
	\begin{enumerate}
		\item\label{gct0dim1} $R$ is of graded finite \CM type;
		\item\label{gct0dim2} $R$ is of graded countable \CM type;
		\item\label{gct0dim3} $R$ is a hypersurface ring.
	\end{enumerate}
\end{prop}
\begin{proof}
	The implication \eqref{gct0dim1} implies \eqref{gct0dim2} is straight forward.  To show \eqref{gct0dim2} implies \eqref{gct0dim3}, assume that $R$ is not a hypersurface.  Thus there must be two linear forms $a,b \in \m \setminus \m^2$ that are basis elements of $\m/\m^2$.  By \cite[Lemma 4.1]{stoneSS}, there are uncountably many distinct homogeneous ideals $\{I_\alpha\}_{\alpha\in \kk}$ in $R$. In this context, we have that $I_\alpha = (a+\alpha b)R$.  Consider the graded indecomposable maximal \CM modules $\{R/I_\alpha\}_{\alpha\in \kk}$.  As each of these modules have different annihilators, we know they are not isomorphic.  A contradiction as we assumed that $R$ was of graded countable type.

	To prove that \eqref{gct0dim3} implies \eqref{gct0dim1}, we consider the $\m$-adic completion of $R$ and then apply \cite[Satz 1.5]{herzog78} to see that the completion is of finite \CM type.  By Remark \ref{rem:local-to-graded}, we know that $R$ also has graded finite \CM type.
\end{proof}

We thus have a complete classification of graded countable \CM type for zero dimensional standard graded rings.  Further, Proposition \ref{prop:gct-0-dim} gives a positive answer to Question \ref{q2} for zero dimensional standard graded rings with uncountable residue field.

\section{One Dimensional Rings}\label{dim1}

In the one-dimensional case, Question \ref{q2} has a positive answer as shown by R. Karr and R. Wiegand \cite[Theorem 1.4]{karr11}. In this section, we examine the Drozd-Ro\u\i ter conditions and give a classification of one dimensional reduced standard graded rings of graded countable \CM type. Thus retrieving the result of R. Karr and R. Wiegand.

\subsection{Finite Type and the Drozd-Ro\u\i ter conditions}

As detailed by N. Cimen, R. Wiegand, and S. Wiegand \cite{cimen}, if $(R,\m, \kk)$ is a one dimensional, Noetherian, reduced, local \CM ring such that the integral closure of $R$, say $S$, is finitely generated as an $R$-module, then we know precisely when $R$ has finite Cohen-Macaulay type.  This happens when the following conditions occur:
	\begin{enumerate}
		\item[\bf DR1] $S$ is generated by $3$ elements as an $R$-module;
		\item[\bf DR2] the intersection of the maximal $R$-submodules of $S/R$ is cyclic as an $R$-module.
	\end{enumerate}
These conditions are called the \dfn{Drozd-Ro\u\i ter conditions}. It was further shown in \cite[Proposition 1.12]{cimen} that the Drozd-Ro\u\i ter conditions are equivalent to the following:
	\begin{align*}
		\tag{dr1} \dim_\kk(S/\m S) &\ls 3;\\ 
		\tag{dr2} \dim_\kk\left(\frac{R +\m S}{R + \m^2S}\right) &\ls 1.
	\end{align*}

In order to have a different grasp of what it means to satisfy the Drozd-Ro\u\i ter conditions, we state another set of equivalent conditions.  This result is stated in the next proposition where $\lambda$ represents length as an $R$-module, and $\ol *$ is the integral closure of ideals.

\begin{prop}\label{prop:ssDR}
	Let $(R,\m,\kk)$ be a one-dimensional, Noetherian, reduced, local \CM ring with finite integral closure and uncountable residue field $\kk$. Let $x$ be a minimal reduction of the maximal ideal $\m$.  The Drozd-Ro\u\i ter conditions are equivalent to the following:
	\begin{align}
		 e(R) &\ls 3; \label{eq:ssDR1} \\ 
		 \lambda(\ol{\m^2}/x\m) &\ls 1. \label{eq:ssDR2}
	\end{align}	
\end{prop}
\begin{proof}
	To show that \eqref{eq:ssDR1} holds, we will show that $e(R) = \dim_\kk(S/\m S)$ where $S$ is the integral closure of $R$.  As $x$ is a reduction of $\m$, we know that $xS$ is also a reduction of $\m S$.  But this holds if and only if $\m S \subseteq \ol{xS}$.  As principal ideals are integrally closed in $S$, we have that
	\[
		xS \subseteq \m S \subseteq \ol{xS} = xS,
	\]
and hence $xS = \m S$.  By assumption, $S$ is finitely generated as an $R$-module.  Therefore we have that the map $S \to S$  defined by multiplication by $x$ is an injection.  Hence, we have the following commutative diagram in Figure \ref{fig}.

\begin{figure}[h!]
\[
	\xymatrixrowsep{5mm}
	\xymatrixcolsep{8mm}
	\xymatrix
		{
			& & 0 \ar[d] & 0 \ar[d] & K \ar[d]  & &  \\
			& 0 \ar[r] & R \ar[r]^{\cdot x} \ar[d] & R \ar[r] \ar[d] & R/xR \ar[r] \ar[d] & 0 &\\
			& 0 \ar[r] & S \ar[r]^{\cdot x} \ar[d] & S \ar[r] \ar[d] & S/xS \ar[r] \ar[d] & 0 &\\
			& & C'  & C'  & C &  &\\
		}
\]
\caption{\label{fig} A commutative diagram where $K$, $C'$, and $C$ are the respective kernel and cokernels.}
\end{figure}

Consider the exact sequence in the right side of Figure \ref{fig}, 
\begin{equation}\label{snake}
	\xymatrixrowsep{5mm}
	\xymatrixcolsep{8mm}
	\xymatrix
		{
			0 \ar[r] & K \ar[r] & R/xR \ar[r] & S/xS \ar[r] & C \ar[r] & 0.
		}
\end{equation}
Notice by the Snake Lemma applied to Figure \ref{fig} that $\lambda(K) = \lambda(C)$; here $\lambda$ represents length as $R$-modules.  
As $C$ and $K$ have the same length, \eqref{snake} shows that $R/xR$ and $S/xS$ also have the same length.  Since $x$ is a minimal reduction of the maximal ideal, we know that $e(R) = \lambda(R/xR)$.  Therefore,
\[
	e(R) = \lambda(R/xR) = \lambda(S/xS) = \lambda(S/\m S) = \dim_\kk(S/\m S).
\]	
	In order to show \eqref{eq:ssDR2}, first notice that
	\begin{equation}\label{eq:iso-ssDR}
		\frac{R+\m S}{R+\m^2 S} \simeq \frac{\m S}{\m^2S + (R\cap \m S)} \simeq \frac{\m S}{\m^2S + \m} \simeq \frac{xS}{x^2S + \m}.
	\end{equation}
For simplicity, we define $B$ as follows,
	\[
		B := \dim_\kk\left(\frac{R +\m S}{R + \m^2S}\right) = \lambda\left(\frac{xS}{x^2S + \m}\right).
	\]
We now consider the short exact sequence
	\[
		\xymatrix
			{
				0 \ar[r] & \ds \frac{\m^2S + \m}{\m^2S} \ar[r] & \ds \frac{\m S}{\m^2S} \ar[r] & \ds \frac{\m S}{\m^2S + \m}  \ar[r] & 0.
			}
	\]
Rewriting the two terms on the left gives us
	\begin{align}
		\frac{\m^2S + \m}{\m^2S} &\simeq \frac{\m}{\m^2 S \cap \m} \simeq \m /\ol{\m^2}; \label{eq:iso-ssDR2}\\
		\frac{\m S}{\m^2S} &\simeq \frac{S}{\m S} = \frac{S}{xS}. \label{eq:iso-ssDR3}
	\end{align}
Combining \eqref{eq:iso-ssDR2} and \eqref{eq:iso-ssDR3} with the above short exact sequence yields
	\begin{equation}\label{eq:length2}
		\lambda\left(\frac{S}{xS}\right) = \lambda\left(\m/\ol{\m^2}\right) + B.
	\end{equation}
On the other hand, consider the following short exact sequence 		
	\begin{equation}\label{eq:ses-length}
		\xymatrix
			{
				0 \ar[r] & \ds \frac{\ol{\m^2} + xR}{xR} \ar[r] & \ds \frac{R}{xR} \ar[r] & \ds \frac{R}{\ol{\m^2} + xR} \ar[r] & 0,
			}
	\end{equation}
along with the isomorphisms
	\begin{equation}\label{equal}
		\frac{\ol{\m^2} + xR}{xR} \simeq \frac{\ol{\m^2}}{xR\cap\ol{\m^2} } = \frac{\ol{\m^2}}{x\m}.
	\end{equation}
Note that the equality in \eqref{equal} can be justified as follows.  Since $xS = \m S$, we know that $\ol{\m^2} = \m^2 S\cap R = x^2S \cap R$.  If $y \in xR\cap\ol{\m^2}$, then $y = xr \in \ol{\m^2} = x^2S\cap R$.  This forces $r \in xS \cap R = \m$.  Hence $y \in x\m$.  Equality follows as $x\m \subseteq xR\cap\ol{\m}$.

Computing length in \eqref{eq:ses-length} gives us
	\[
		\lambda\left(\frac{R}{xR}\right) =  \lambda\left( \frac{\ol{\m^2}}{x\m}\right) +  \lambda\left( \frac{R}{\ol{\m^2} + xR}\right).
	\]

We can repeat the above steps with the following short exact sequence and isomorphisms:
	\[
		\xymatrix
			{
				0 \ar[r] & \ds  \frac{\ol{\m^2} + xR}{\ol{\m^2}} \ar[r] & \ds \frac{~R~}{\ol{\m^2}} \ar[r] & \ds \frac{R}{\ol{\m^2} + xR} \ar[r] & 0;
			}
	\]
	\[
		\frac{\ol{\m^2} + xR}{\ol{\m^2}} \simeq \frac{xR}{\ol{\m^2}\cap xR} \simeq \frac{xR}{\m xR} \simeq \frac{R}{\m}.
	\]
Once again, if we compute the length, we have that
	\begin{equation}\label{eq:length3}
		\lambda\left(\frac{~R~}{\ol{\m^2}}\right) =  1 +  \lambda\left( \frac{R}{\ol{\m^2} + xR}\right)
	\end{equation}
Combining \eqref{eq:length2} and \eqref{eq:length3} with the fact that $\lambda\left(\frac{R}{xR}\right) = \lambda\left(\frac{S}{xS}\right)$ and $\lambda\left(\frac{~\m~}{\ol{\m^2}}\right) = \lambda\left(\frac{~R~}{\ol{\m^2}}\right) - 1$, we have
	\begin{align*}
		\lambda\left( \frac{\ol{\m^2}}{x\m}\right) +  \lambda\left( \frac{R}{\ol{\m^2} + xR}\right) & = \lambda\left(\frac{~\m~}{\ol{\m^2}}\right) + B \\
					&= \lambda\left(\frac{~R~}{\ol{\m^2}}\right) - 1 + B \\
					&= 1 +  \lambda\left( \frac{R}{\ol{\m^2} + xR}\right) - 1 + B.
	\end{align*}
Simplifying we see that $\ds \lambda\left( \frac{\ol{\m^2}}{x\m}\right) = B$.  We now have
	\begin{align*}
		e(R) & = \dim_\kk(S/\m S); \\
		\lambda\left( \frac{\ol{\m^2}}{x\m}\right) & = \dim_\kk\left(\frac{R +\m S}{R + \m^2S}\right).
	\end{align*}
Applying (dr1) and (dr2) yields the desired result.	
\end{proof}

Given Proposition \ref{prop:ssDR}, we can construct a couple of examples.  

\begin{example}
	Consider the ring $R = \kk\llbracket t^3,t^7 \rrbracket$. This is a one-dimensional domain with $e(R)=3$.  The element $t^3$ is a minimal reduction of the maximal ideal $(t^3,t^7)R$.  If we compute the length, we see that $\lambda \left( \ol{\m^2} /t^3\m \right)  = 2$.  Hence, by Proposition \ref{prop:ssDR}, we have that $R$ is not of finite type.
\end{example}

\begin{example}
	Let $R = \kk\llbracket x,y\rrbracket/(x^3y-xy^3)$.  This ring is one-dimensional and reduced.  If we compute the multiplicity, we find that $e(R) = 4$.  Thus, we immediately have from Proposition \ref{prop:ssDR} that $R$ is not of finite type.  Computing the length none-the-less, we find that $\lambda\left(\ol{\m^2}/(x+2y)\m\right)  = 1$.  
\end{example}

\subsection{Graded Countable Type}

	In \cite[Theorem 5.3]{stoneSS}, one dimensional standard graded rings of graded countable type were shown to either be of minimal multiplicity or a hypersurface.  Turning to the case of minimal multiplicity, we find the following.

\begin{thm}\label{thm:1n}
	Let $(R,\m,\kk)$ be a standard graded one dimensional \CM ring with uncountable residue field $\kk$. If the $h$-vector of $R$ is $(1,n)$ with $n \gs 3$, then $R$ is not of graded countable \CM type.
\end{thm}
\begin{proof}
Let $x$ be a minimal homogeneous reduction of the maximal ideal $\m$, and let $x,u,v,w,c_4,\cdots,c_n\in \m$ be elements of a minimal $k$-basis of $\m/\m^2$. By assumption $n \gs 3$, so we are guaranteed at least four elements in the basis of $\m/\m^2$.  Without losing any generality, we assume that $w$ is the fourth basis element.   Assume that there is a graded isomorphism
\[
	I_\alpha = ( x, u+ \alpha v ) \simeq ( x, u + \beta v ) = I_\beta
\]
where $\alpha,\beta$ are elements of $\kk$. As $\dim(R)=1$, these ideals are graded indecomposable maximal \CM modules.  Since this isomorphism is graded of degree 0, we have that
\begin{align*}
	x & \mapsto \delta_1 x + \delta_2 (u + \beta v) \\
	u + \alpha v & \mapsto \delta_3 x + \delta_4 (u + \beta v ) 
\end{align*}
where $\det (\delta_i)$ is a unit and $\delta_i$ are elements of $\kk$.  We have that
\begin{equation}\label{eqn:HS13_Original}
	0 = \delta_3 x^2 + \delta_4 x (u + \beta v ) - \delta_1 x(u + \alpha v) - \delta_2 (u + \alpha v )(u + \beta v ).
\end{equation}
Notice that $(u + \alpha v )(u + \beta v )$ is an element of $\m^2$.  As $R$ is of minimal multiplicity, we have that $x\m = \m^2$.  Hence we can view elements of $\m^2$ as elements of $x\m$.  In particular we view $u^2,uv,v^2$ in the following way
\[
	\begin{pmatrix} 
		u^2 \\
		uv \\
		v^2 \\
	\end{pmatrix} = \begin{pmatrix}
		x( {a}_{10}x + {a}_{11}u + {a}_{12}v + {a}_{13}w + a_{15}c_4 +\cdots + a_{1n}c_n )\\
		x( {a}_{20}x + {a}_{21}u + {a}_{22}v + {a}_{23}w + a_{25}c_4 +\cdots + a_{2n}c_n )\\
		x( {a}_{30}x + {a}_{31}u + {a}_{32}v + {a}_{33}w + a_{35}c_4 +\cdots + a_{3n}c_n )\\
	\end{pmatrix} = x \cdot A \cdot \begin{pmatrix}
		x \\
		u \\
		v \\
		w \\
		c_4 \\
		\vdots \\
		c_n
	\end{pmatrix}
\]  
where the matrix $A = (a_{ij})$, $1\leqslant i \leqslant 3$, $0\leqslant j \leqslant n$.  Since $u^2,uv,v^2$ are homogeneous elements, the grading forces the entries of $A$ to be elements of $\kk$.  
Further, if we let 
\[
	\Phi = \begin{pmatrix} 1 & \alpha+\beta & \alpha \beta \end{pmatrix} \text{ and\ \ } \bb = \begin{pmatrix} x & u & v & w & c_4 & \cdots & c_n \end{pmatrix}^t,
\]
then we can use matrix notation to write
\[
	(u + \alpha v )(u + \beta v ) = u^2 +(\alpha + \beta)uv + \alpha\beta v^2 = x \cdot \Phi \cdot A \cdot \bb.
\]
We can now cancel the $x$ in equation (\ref{eqn:HS13_Original}) and rewrite it as
\begin{equation}\label{eq:coeff-matrix}
	0 = \begin{pmatrix}
			\delta_3 - \Phi A_0 \delta_2 \\ \delta_4-\delta_1 - \Phi A_1 \delta_2 \\ \beta\delta_4 - \alpha\delta_1 - \Phi A_2 \delta_2 \\ -\Phi A_3 \delta_2 \\ -\Phi A_4 \delta_2 \\ \vdots \\ -\Phi A_n \delta_2
	\end{pmatrix}^t \cdot \bb
\end{equation}
where $A_i$ are the columns of the matrix $A$. All of the elements in the coefficient matrix of \eqref{eq:coeff-matrix} are elements of $k$ and hence equal zero as $x,u,v,w,c_4,\ldots,c_n$ form a $\kk$-basis. 

	At this point we focus on $\delta_2$.  If $\delta_2 \not=0$, then the fact that $\Phi A_3 \delta_2 = 0$ implies that
	\begin{equation}\label{eq:1n-d2=0}
		a_{13}+(\alpha + \beta)a_{23} +\alpha\beta a_{33} = 0
	\end{equation}
in the field $\kk$.  As the $a_{ij}$ are independent of our choice of $\alpha$ and $\beta$, Equation \eqref{eq:1n-d2=0} shows that every $\alpha,\beta$ such that $I_\alpha \simeq I_\beta$ is a solution to 
\[
	f(X,Y) = a_{13}+(X + Y)a_{23} +XY a_{33} \in \kk[X,Y].
\]
This forces $f(X,Y)$ to be identically zero, a contradiction.  Hence there are uncountably many $I_\alpha$ that are not isomorphic.

 	If we let $\delta_2 = 0$, then Equation \eqref{eq:1n-d2=0} becomes the relation
	\begin{equation}\label{eqn:IdealRelation-delta=0}
		\delta_3 x + (\delta_4 - \delta_1)u + (\beta \delta_4 - \alpha \delta_1) v = 0.
	\end{equation}
As $x,u,v$ are independent over $\kk$, we have that the coefficients are zero.  In particular $\delta_4 - \delta_1 = 0$. Since $\delta_1\delta_4-\delta_2\delta_3 \not= 0$, we know that $\delta_1 = \delta_4 \not=0$. Thus the fact that $\beta \delta_4 - \alpha \delta_1 = 0$ implies that $\alpha = \beta$.  Hence there are uncountably many non-isomrophic ideals $I_\alpha$.
\end{proof}

Given the above results, we are now ready to characterize one dimension standard graded rings of graded countable \CM type.

\begin{cor}\label{cor:dim1-class}
	Let $(R, \m,\kk)$ be a one-dimensional standard graded \CM ring with uncountable residue field $\kk$.  If $R$ is of graded countable \CM type, then $R$ is either of minimal multiplicity with $h$-vector $(1,2)$, or is isomorphic to one of the following hypersurfaces:
	\begin{enumerate}
		\item $\kk[x]$; 
		\item $\kk[x,y]/(xy)$; 
		\item $\kk[x,y]/(xy(x+y))$;
		\item $\kk[x,y]/(xy^2)$;
		\item $\kk[x,y]/(y^2)$.
	\end{enumerate}
\end{cor}
\begin{proof}
	A direct application of \cite[Theorem 5.3]{stoneSS} and Theorem \ref{thm:1n} show that $R$ is either a hypersurface ring, or has minimal multiplicity with $h$-vector $(1,2)$.

	Concerning the hypersurfaces, items (1)-(3) have graded finite \CM type as can be seen from \cite{buchweitz87} or \cite{eisenbud88}.  The hypersurfaces (4) and (5) are not graded finite \CM type, but their completions are the one dimensional $(A_\infty)$ and $(D_\infty)$ hypersurface singularities shown in \eqref{eq1} and \eqref{eq2} on page \pageref{eq1}.  It was shown by R. Buchweitz, G. Greuel, and F. Schreyer in \cite{buchweitz87} that these are the only hypersurfaces that are countable but not finite \CM type. Hence by Remark \ref{rem:local-to-graded}, the rings (4) and (5) are of graded countable \CM type.
\end{proof}

\begin{cor}\label{cor:dim1-gct-mult}
	Let $(R, \m, \kk)$ be a one-dimensional standard graded \CM ring with uncountable residue field $\kk$. If $R$ is of graded countable \CM type then $e(R) \ls 3$.
\end{cor}
\begin{proof}
	As shown in \cite[Corollary 4.5]{stoneSS}, the possible $h$-vectors are $(1)$, $(1,n)$, or $(1,n,1)$ for some integer $n$.  Combining this with Corollary \ref{cor:dim1-class}, we know that the possible $h$-vectors of $R$ are the following:
	\[
		(1),\ (1,1),\ (1,2),\ (1,1,1).
	\]
Hence we have that $e(R) \ls 3$.
\end{proof}

An obvious improvement to Corollary \ref{cor:dim1-class} would be to classify the rings whose $h$-vector is $(1,2)$.  Currently, there are two known examples of one-dimensional standard graded \CM rings of graded countable \CM type having $h$-vector $(1,2)$:
\begin{align}
	& \kk[x,y,z]/(xy,yz,z^2); \label{eqn:gw-1,2} \\
	& 	\kk[x,y,z] /{\det}_2\begin{pmatrix}
		 									x & y & z \\
											y & z & x 
								\end{pmatrix}.  \label{eqn:graded-1,2} 
\end{align}
The completion of \eqref{eqn:gw-1,2} is the local endomorphism ring of $(x,y)R$ where $R = \kk\llbracket x, y\rrbracket /( xy^2 )$, the one-dimensional $(D_\infty)$ hypersurface.  This endomorphism ring is known to be of countable \CM type (see \cite[Theorem 1.5]{leuschke05} or \cite[Example 14.23]{leuschke}) and thus by Remark \ref{rem:local-to-graded}, \eqref{eqn:gw-1,2} is of graded countable \CM type.  It is worth noting that \eqref{eqn:gw-1,2} is not reduced.

	As for \eqref{eqn:graded-1,2}, this is the only one dimensional ring of finite type and $h$-vector $(1,2)$ found in the list of graded finite \CM type rings \cite{eisenbud88}.  Given this result, if we assume our ring to have an isolated singularity, then we have a positive answer to Question \ref{q2}.

\begin{cor}\label{1reduce}
	Let $(R, \m, \kk)$ be a one-dimensional, reduced, standard graded \CM ring with uncountable residue field $\kk$.  Further assume that $R$ has a finite integral closure. If $R$ is of graded countable \CM type, then $R$ is of graded finite type and isomorphic to one of the following:
	\begin{enumerate}
		\item $\kk[x]$; 
		\item $\kk[x,y]/(xy)$; 
		\item $\kk[x,y]/(xy(x+y))$;
		\item \(
			\kk[x,y,z]/{\det}_2\begin{pmatrix}
			 									x & y & z \\
												y & z & x 
									\end{pmatrix}.
		\)
	\end{enumerate}
\end{cor}
\begin{proof}
	By Remark \ref{rem:local-to-graded}, we can pass to the completion. Since local rings of minimal multiplicity with $h$-vector $(1,2)$ have $\lambda\left(\ol{\m^2}/x\m\right) \ls 1$, we can apply Proposition \ref{prop:ssDR} and Corollary \ref{cor:dim1-class} to obtain the desired result.
\end{proof}

\section{Non-Gorenstein Rings of Dimension at least Two} \label{sec:dim-least-2}

We are unable to classify standard graded rings of graded countable type with dimension larger than one.  However, if we assume an isolated singularity, we have a positive answer to Question \ref{q2} in the non-Gorenstein case.

\subsection{Non-Gorenstein Rings of Dimension Two} \label{sec:dim2}

	As shown in \cite{eisenbud88}, two dimensional standard graded \CM rings with an isolated singularity must be a domain and have minimal multiplicity.  In \cite{eisenbud-harris}, it can be seen that standard graded \CM domains of minimal multiplicity have the following classification: 
	\begin{enumerate} 
		\item[(i)] quadratic hypersurfaces (not necessarily nondegenerate);
		\item[(ii)] cones over the Veronese embedding $\mathbb P^2 \hookrightarrow \mathbb P^5$.  I.e., rings isomorphic to 
			\[
				\kk[x_0,\lst x n]\Big / {\det}_2\begin{pmatrix}
			 						x_0 & x_1 & x_2 \\
									x_1 & x_3 & x_4 \\ 
									x_2 & x_4 & x_5
								  \end{pmatrix},
			\]
			where for $n\gs 5$;

		\item[(iii)] rational normal scrolls.
	\end{enumerate}
	The ring defined in (i) is Gorenstein and is discussed in Section \ref{sec:gor}.  The ring in (ii) is of graded finite (hence countable) \CM type when $n=5$ as can be seen in \cite{eisenbud88} or \cite[Example 17.6.1]{yoshino90}.  For $n >6$, the dimension of the singular locus is larger than 1, hence they are not of graded countable type. 
	However when $n=6$, it is unclear if the ring is graded countable type or not. Notice that this case does not concern Question \ref{q2} as the ring does not have an isolated singularity, but does play a role in a general classification. Further, the dimension of each ring in (ii) is at least three.  As such, it is left to consider the rational normal scrolls defined as follows.
	\begin{defn}\label{def:scroll}
		Let $0 \ls a_0 \ls a_1 \ls \cdots \ls a_k$ be given integers and let $\{x_{j}^i \vl 0 \ls j \ls a_i, 0\ls i \ls k \}$ be a set of variables over a field $\kk$.  Then, take the ideal $I$ of the polynomial ring $S = \kk[x_{j}^i \vl 0 \ls j \ls a_i, 0 \ls i \ls k]$ generated by all the $2\times 2$-minors of the matrix:
		\[
			\begin{pmatrix}
				x_{0}^0 & x_{1}^0 & \cdots & x_{a_0-1}^0  & | &  \cdots & | \ x_{0}^k  & x_{1}^k & \cdots & x_{a_k-1}^k \\
				x_{1}^0 & x_{2}^0 & \cdots & x_{a_0}^0   & |  & \cdots & |\  x_{1}^k & x_{2}^k & \cdots & x_{a_k}^k
			\end{pmatrix}.
		\]
	Define the graded ring $R$ to be the quotient $S/I$ with $\deg(x_{j}^i) = 1$ for all $i,j$, and call $R$ the \dfn{scroll} of type $(a_0,\lst a k)$.  
	\end{defn}

		Notice that the polynomial ring $\kk[x_0, \lst x N]$ is a rational normal scroll of type $(0,\cdots, 0 , 1)$, where $N = \sum_{i=0}^k a_i +k+1$. Further, a scroll is of type $(0,\cdots,0,1)$ if and only if it is regular (for basic properties of scrolls, see \cite{eisenbud-harris,miro13}).  In general, it is known that scrolls are integrally closed \CM domain of dimension $k+2$. 

\begin{prop}\label{prop:dim2-nonGor}
	Let $(R,\m,\kk)$ be a standard graded \CM ring that is not Gorenstein and $\dim(R) = 2$. Further assume that $R$ has an isolated singularity  and that $\kk$ is an uncountable field. If $R$ is of graded countable \CM type, then $R$ is of graded finite type and is isomorphic to
	\[
		\kk[x_0, \lst x n] \Big /{\det}_2\begin{pmatrix}
		 									x_0 & \ldots & x_{n-1} \\
											x_1 & \ldots & x_{n} 
									\end{pmatrix},
	\]
where $n\gs 2$.
\end{prop}
\begin{proof} 
	It was shown in \cite{eisenbud88} that two dimensional, non-Gorenstein rings with an isolated singularity have minimal multiplicity.

By the above discussion, we know that $R$ must be isomorphic to a two dimensional scroll.  The only non-Gorenstein scrolls of dimension two are of type $(m)$ where $m\gs 2$.  These are the rings listed in the statement of Proposition \ref{prop:dim2-nonGor}.  It is known that these rings are of graded finite \CM type (see \cite{eisenbud88} or \cite[Theorem 2.3]{auslander87}).
\end{proof}

	Proposition \ref{prop:dim2-nonGor} gives a partial (positive) answer to Question \ref{q2}. Ideally though, we would like to remove the isolated singularity condition from the hypothesis of Proposition \ref{prop:dim2-nonGor}.  Doing so would show that graded countable \CM type implies graded finite \CM type for non-Gorenstein two dimensional rings.

	\subsection{Non-Gorenstein Rings of Dimension at least Three}\label{sec:dim3}

	Let $(R,\m,\kk)$ be a standard graded \CM ring of graded countable \CM type, that is not Gorenstein and $\dim R \gs 3$.  By \cite[Proposition 4.6]{stoneSS}, we know that $R$ must be a domain and have minimal multiplicity.  Thus we only have to consider the three classes of rings listed in Section \ref{sec:dim2} in order to obtain the following.

	\begin{prop}\label{prop:dim3-nongor-gct}
		Let $(R,\m,\kk)$ be a standard graded \CM ring with uncountable residue field $\kk$ and an isolated singularity. Further assume that $R$ is not Gorenstein and $\dim R \gs 3$.  If $R$ is of graded countable \CM type, then $R$ is of graded finite type and is isomorphic one of the following rings:
		\begin{enumerate}
			\item $\kk[x_1,\ldots,x_{5}]/{\det}_2\begin{pmatrix}
		 									x_1 & x_2 & x_4 \\
											x_2 & x_3 & x_5 
									\end{pmatrix}$;
			\item $\kk[x_1,\ldots,x_{6}]/{\det}_2(\text{sym}\ 3\times 3)$.
		\end{enumerate}
	\end{prop}
	\begin{proof}

		Consider the classes of rings listed in Section \ref{sec:dim2}.  We can ignore the rings in (i) as they are Gorenstein.  The rings in (ii) only have an isolated singularity when $n = 5$.  This ring is known to be of graded finite \CM type \cite{eisenbud88}. As such, we list it above.

		Concerning the scrolls in (iii), M. Auslander and I. Reiten  classify the scrolls of finite type in \cite{auslander89b}.  In particular, they show that if $k \gs 1$ and $R$ is not of type $(1,1)$ or $(1,2)$, then $R$ has $|\kk|$ many indecomposable graded \CM modules, up to shifts.  As it is, the graded scrolls of type $(1,1)$ and $(1,2)$ are the only graded scrolls of dimension at least 3 that have graded countable \CM type. Notice that a scroll of type $(1,1)$ is Gorenstein and hence we are left with the scroll of type $(1,2)$.
	\end{proof}

	Proposition \ref{prop:dim3-nongor-gct} gives another case where Question \ref{q2} has a positive answer. Concerning the general classification, in order to remove the isolated singularity condition from the hypothesis, one only needs to determine the graded \CM  type of 
		\begin{equation}\label{eqn:veronese-cone}
			\kk[x_0,\lst x 6]\Big / {\det}_2\begin{pmatrix}
		 						x_0 & x_1 & x_2 \\
								x_1 & x_3 & x_4 \\ 
								x_2 & x_4 & x_5
							  \end{pmatrix}.
		\end{equation}
	Knowing this information would give a complete classification for non-Gorenstein rings of dimension at least three.

	\section{Gorenstein Rings}\label{sec:gor}

	Not much is known about Gorenstein rings of graded countable \CM type. However, it is well known that standard graded Gorenstein rings of graded finite \CM type are hypersurfaces \cite[Satz 1.2]{herzog78}.  This fact is heavily exploited in the classification of standard graded \CM rings of graded finite \CM type \cite{eisenbud88}.  The countable analog of this fact is still unknown.

	\begin{conj}\label{conj:gor}
		A Gorenstein ring of countable \CM type is a hypersurface.
	\end{conj}

	Using the concept of \SuSp, Conjecture \ref{conj:gor} was shown to be true for standard graded rings of dimension at most one \cite{stoneSS}, but the conjecture remains open for higher dimensions. By the results of \cite{knorrer87, buchweitz87} on hypersurfaces, a positive result to Conjecture \ref{conj:gor} would completely classify Gorenstein rings of graded countable \CM type.

	If we restrict to rings of minimal multiplicity, then we are only dealing with hypersurfaces, as $h$-vectors of Gorenstein rings are symmetric. As such, by \cite{knorrer87, buchweitz87} we have a complete characterization.  In particular, consider a standard graded Gorenstein ring $(R,\m,\kk)$ of minimal multiplicity, i.e. the $h$-vector is $(1,1)$. Hence if $R$ is of graded countable \CM type, then $R$ is isomorphic to one of the following hypersurfaces:
		\begin{align*}
			(A_1): &\  \kk[\lst x n]/(x_1^2 + \cdots + x_n^2),\ n \gs 1; \\
			(A_\infty): &\  \kk[\lst x n]/(x_2^2 + \cdots + x_n^2),\  n \gs 1.
		\end{align*}
	\begin{rem}\label{cor:gor-iso-min-finite}
	 	If we further assume that a standard graded Gorenstein ring of minimal multiplicity has an isolated singularity, then graded countable \CM type implies graded finite \CM type. To see this, notice that only $(A_1)$ has a zero dimensional singular locus.
	\end{rem}

	It is worth noting that standard graded rings of dimension at least 2 and with graded finite \CM type have minimal multiplicity. This can be seen by examining the list graded rings of graded finite \CM type in \cite{eisenbud88}.  With these results, we ask the following question.

	\begin{quest}\label{quest:gor-min} 
		If $(R,\m,\kk)$ is a standard graded Gorenstein ring of graded countable \CM type and $\dim(R) \gs 2$, is $R$ necessarily of minimal multiplicity?
	\end{quest}

		A positive answer to Question \ref{quest:gor-min} would ultimately force the ring to be a hypersurface and affirm Conjecture \ref{conj:gor}. Further it would show that Question \ref{q2} has an affirmative answer for Gorenstein rings of graded countable \CM type.   

	\begin{rem}
			Since Conjecture \ref{conj:gor} is still open, the natural place to look are rings of complete intersection.  According to \cite{stoneSS}, we know that standard graded complete intersection with graded countable \CM type are either a hypersurface or defined by two quadrics.  With this fact, one could use tools from representation theory \cite{stevenson} and the analysis of the $h$-vector, to show that the ring is indeed a hypersurface no matter what the dimension. As these techniques are outside the scope of the paper, we leave it here as a remark.
	\end{rem}

	\section{Summary of Results}\label{sec:sumup}

	In Corollary \ref{main:cor} we summarize the previous results in relation to Question \ref{q2}. As noted earlier, proving Conjecture \ref{conj:gor} would give a complete answer to the question.  In the rest of the section, we consider what can be shown when the ring in question does not necessarily have an isolated singularity.  In particular, we state a partial classification of standard graded rings of graded countable \CM type.  

	\begin{cor}\label{main:cor}
		Let $(R,\m,\kk)$ be a standard graded ring of graded countable \CM type.  Further assume that $R$ has an isolated singularity and that $\kk$ is an uncountable algebraically closed field of characteristic zero.  Then $R$ is of graded finite type if one of the following hold:
		\begin{enumerate}
			\item $\dim(R) \ls 1$;
			\item $R$ is not Gorenstein;
			\item $R$ is of minimal multiplicity.
		\end{enumerate} 
	\end{cor}
	\begin{proof}
	If the dimension of $R$ is at most 1, then we can directly apply Proposition \ref{prop:gct-0-dim} and Corollary \ref{1reduce}. For non-Gorenstein rings of dimension at least 2, we can apply Propositions \ref{prop:dim2-nonGor} and \ref{prop:dim3-nongor-gct}.  If $R$ is Gorenstein of minimal multiplicity and $\dim(R)\gs 2$, then we have graded finite \CM type from Remark \ref{cor:gor-iso-min-finite} and the discussion preceding it.
	\end{proof}

	\subsection{Partial classification in the general case}

	Throughout this paper, we have obtained a partial classification of standard graded rings of graded countable \CM type. For uniformity's sake, we state the partial classification in the same format as the classification of graded finite \CM type found in \cite{eisenbud88}. As such, notice that rings listed in the arbitrary dimension section are omitted from the other cases.  In particular, the $A_\infty$ hypersurface is only listed in the arbitrary dimension section.  For more details on rings of countable \CM type, see \cite{leuschke}.\smallskip

	\paragraph{\bf Arbitrary Dimension:} The following rings are shown to be graded countable type in \cite{eisenbud88, knorrer87, buchweitz87}.  This list is not known to be complete.
	\begin{align*} 
		& \kk[\lst x n],\ n \gs 1; \\
		& \kk[\lst x n]/(x_1^2 + \cdots + x_n^2),\ n \gs 1; \\
		& \kk[\lst x n]/(x_2^2 + \cdots + x_n^2),\ n \gs 1.	
	\end{align*}

	\paragraph{\bf Dimension Zero:} As seen in Proposition \ref{prop:gct-0-dim}, graded countable type is the same as graded finite type.  By \cite{eisenbud88}, we have the following complete list.
	\begin{align*}
		& \kk[x]/(x^m), \ m \gs 1.
	\end{align*}

	\paragraph{\bf Dimension One:} 
	We do not have a complete list of dimension one graded countable type rings. However, by Corollary \ref{1reduce} and the discussion before it, we know these rings must have $h$-vector $(1,2)$ or be isomorphic to
	\begin{align*}
		& \kk[x,y]/(xy); \\
		& \kk[x,y]/(xy(x+y));\\
		& \kk[x,y]/(xy^2);\\
		& \kk[x,y,z]/{\det}_2\begin{pmatrix}
			 									x & y & z \\
												y & z & x 
									\end{pmatrix}; \\
		& \kk[x,y,z]/(xy,yz,z^2).
	\end{align*}

	\paragraph{\bf Dimension Two:} According to Proposition \ref{prop:dim2-nonGor}, the non-Gorenstein graded countable type rings with an isolated singularity are known.  Besides the ring given in the Arbitrary Dimension section, we are not aware of any other graded countable type ring with a non-isolated singularity. Thus we state what was already recorded in \cite{eisenbud88}.
	\begin{align*}
		& \kk[x_1,\ldots,x_{n+1}]\Big /{\det}_2\begin{pmatrix}
		 									x_1 & \ldots & x_n \\
											x_2 & \ldots & x_{n+1} 
									\end{pmatrix}, \ n \gs 2.
	\end{align*}

	\paragraph{\bf Dimension Three:} 
	If we restrict the proof of Proposition \ref{prop:dim3-nongor-gct} to three dimensional rings, we can omit the isolated singularity condition. As such, we know the rings below are the only non-Gorenstein graded countable type rings of dimension three.  As the Gorenstein case is open, it is not know if this is a complete list or not. 
	\begin{align*}
		& \kk[x_1,\ldots,x_{5}]\Big /{\det}_2\begin{pmatrix}
	 									x_1 & x_2 & x_4 \\
										x_2 & x_3 & x_5 
								\end{pmatrix};\\
		& \kk[x_1,\ldots,x_{6}]/{\det}_2(\text{sym}\ 3\times 3).
	\end{align*}

	\begin{rem}
		According to the remarks after Proposition \ref{prop:dim3-nongor-gct}, for non-Gorenstein rings of dimension at least four, the only possible ring is \eqref{eqn:veronese-cone}.
	\end{rem}

\section{Acknowledgements}
	
	The author would like to thank the anonymous referee for valuable feedback and suggestions regarding the content of this paper. For instance, the referee pointed out the existence of the ring \eqref{eqn:gw-1,2} as well as some discrepancies in the initial version of Section~\ref{sec:dim-least-2}. Further, the author is especially thankful to Craig Huneke for several useful conversations. Additionally, results in this note were inspired by many Macaulay2 \cite{M2} computations.



\def\cprime{$'$}

\end{document}